\documentclass[12pt]{article}
\usepackage{amsmath,amssymb,amsthm}
\usepackage{bm} 
\usepackage{ascmac} 
\usepackage{amscd} 
\usepackage{stmaryrd} 

\newtheorem{thm}{Theorem}
\newtheorem{crl}{Corollary}
\newtheorem{cnj}{Conjecture}
\newtheorem{prp}{Proposition}
\newtheorem{lmm}{Lemma}
\newtheorem{rmk}{Remark}
\newtheorem{dfn}{Definition}

\newtheorem*{thm*}{Theorem}

\newcommand{\LR}{\Leftrightarrow}
\newcommand{\st}{\stackrel}
\newcommand{\ra}{\rightarrow}
\newcommand{\mt}{\mapsto}

\newcommand{\HFR}{\mathrm{Hom}(F,\mathbb R)}
\newcommand{\HFCp}{\mathrm{Hom}(F,\mathbb C_p)}

\newcommand{\prn}{\mathbb R_{+}}

\newcommand{\nni}{\mathbb Z_{\geq 0}}
\newcommand{\ttt}{\,{}^t}

\renewcommand{\Re}{\mathrm{Re}}

\newcommand{\mfpi}{\mathfrak p_\iota}
\newcommand{\fpi}{\mathfrak f\mathfrak p_\iota}

\newcommand{\efp}{E_{\mathfrak f,+}}

\newcommand{\oq}{\overline{\mathbb Q}}

\makeatletter
\newcommand{\subjclass}[2][2010]{%
  \let\@oldtitle\@title%
  \gdef\@title{\@oldtitle\footnotetext{#1 \emph{Mathematics subject classification(s).} #2}}%
}
\newcommand{\keywords}[1]{%
  \let\@@oldtitle\@title%
  \gdef\@title{\@@oldtitle\footnotetext{\emph{Key words and phrases.} #1.}}%
}
\makeatother

\title{$p$-adic measures associated with zeta values and $p$-adic $\log$ multiple gamma functions}
\author{Tomokazu Kashio\thanks{Tokyo University of Science, \texttt{kashio\_tomokazu@ma.noda.tus.ac.jp}}}
\subjclass{11R27, 11R37, 11R42, 11R80, 11S40, 11S80, 33B15.}
\keywords{the Gross-Stark conjecture, multiple gamma functions, $p$-adic measures}

\begin{document}


\maketitle

\begin{abstract}
We study a relation between two refinements of the rank one abelian Gross-Stark conjecture:
For a suitable abelian extension $H/F$ of number fields, a Gross-Stark unit is defined as a $p$-unit of $H$ satisfying some proporties.
Let $\tau \in \mathrm{Gal}(H/F)$.
Yoshida and the author constructed the symbol $Y_p(\tau)$ by using $p$-adic $\log$ multiple gamma functions,
and conjectured that the $\log_p$ of a Gross-Stark unit can be expressed by $Y_p(\tau)$.
Dasgupta constructed the symbol $u_T(\tau)$ by using the $p$-adic multiplicative integration, 
and conjectured that a Gross-Stark unit can be expressed by $u_T(\tau)$.
In this paper, we give an explicit relation between $Y_p(\tau)$ and $u_T(\tau)$.
\end{abstract}


\section{Introduction}

Let $F$ be a totally real field, $K$ a CM-field which is abelian over $F$, $S$ a finite set of places of $F$.
We assume that 
\begin{itemize}
\item $S$ contains all infinite places of $F$, all places of $F$ lying above a rational prime $p$, and all ramified places in $K/F$.
\item Let $\mathfrak p$ be the prime ideal corresponding to the $p$-adic topology on $F$. (Hence $\mathfrak p \in S$.)
Then $\mathfrak p$ splits completely in $K/F$.
\end{itemize}
For $\tau \in \mathrm{Gal}(K/F)$, we consider the partial zeta function 
\begin{align*}
\zeta_S(s,\tau):=\sum_{(\frac{K/F}{\mathfrak a})=\tau,\ (\mathfrak a,S)=1}N\mathfrak a^{-s}.
\end{align*}
Here $\mathfrak a$ runs over all integral ideals of $F$, relatively prime to any finite places in $S$, 
whose image under the Artin symbol $(\frac{K/F}{*})$ is equal to $\tau$.
The series converges for $\Re(s)>1$, has a meromorphic continuation to the whole $s$-plane, and is analytic at $s=0$.
Moreover, under our assumption, we see that 
\begin{itemize}
\item There exists the $p$-adic interpolation function $\zeta_{p,S}(s,\tau)$ of $\zeta_S(s,\tau)$.
\item $\mathrm{ord}_{s=0}\zeta_S(s,\tau),\mathrm{ord}_{s=0}\zeta_{p,S}(s,\tau) \geq 1$. 
\item There exist a natural number $W$, a $\mathfrak p$-unit $u$ of $K$, which satisfy 
\begin{align} \label{GSu}
\log |u^\tau|_\mathfrak P=-W\zeta_S'(0,\tau) \quad (\tau \in \mathrm{Gal}(K/F)).
\end{align}
Here $\mathfrak P$ denotes the prime ideal corresponding to the $p$-adic topology on $K$, $|x|_\mathfrak P:=N\mathfrak P^{-\mathrm{ord}_\mathfrak Px}$.
\end{itemize}
Gross conjectured the following $p$-adic analogue of the rank $1$ abelian Stark conjecture:

\begin{cnj}[{\cite[Conjecture 3.13]{Gr}}] \label{GSc}
Let $u$ be a $\mathfrak p$-unit characterized by {\rm (\ref{GSu})} up to roots of unity. Then we have
\begin{align*}
\log_p N_{K_\mathfrak P/\mathbb Q_p}(u^\tau)=-W\zeta_{p,S}(0,\tau).
\end{align*}
\end{cnj}

Dasgupta-Darmon-Pollack \cite{DDP} proved a large part of Conjecture \ref{GSc}.
Yoshida and the author, and independently Dasgupta formulated refinements of Conjecture \ref{GSc}:
Let $\mathfrak f$ be an integral ideal of a totally real field $F$ satisfying $\mathfrak p\nmid \mathfrak f$, 
$H_\mathfrak f$ the narrow ray class field modulo $\mathfrak f$, 
$H$ the maximal subfield of $H_\mathfrak f$ where $\mathfrak p$ splits completely.
Yoshida and the author \cite{KY1} essentially constructed the invariant $Y_p(\tau)$ (Definition \ref{Yp}) 
for $\tau \in \mathrm{Gal}(H/F)$ by using $p$-adic $\log$ multiple gamma functions.
Then \cite[Conjecture A$'$]{KY1} states that $\log_p u^\tau$ (without $N_{K_\mathfrak P/\mathbb Q_p}$) can be expressed by $Y_p(\tau)$.
On the other hand, Dasgupta constructed the invariant $u_T(\mathfrak b ,\mathcal D_\mathfrak f)$ (Definition \ref{ueta}-(iv)) 
by using the multiplicative integration for $p$-adic measures associated with Shintani's multiple zeta functions.
Then \cite[Conjecture 3.21]{Da} states that a modified version of $u^\tau$ can be expressed by $u_T(\mathfrak b ,\mathcal D_\mathfrak f)$.
In \cite[Remark 2]{Ka3}, the author announced the following relation between these refinements.

\begin{thm*}[Theorem \ref{main}]
Let $\eta$ be a ``good'' prime ideal in the sense of {\rm Definition \ref{assump}}.
We put $T:=\{\eta\}$. 
Then we have 
\begin{align*}
\log_p(u_\eta(\mathfrak b ,\mathcal D_\mathfrak f))
=-Y_p((\tfrac{H/F}{\mathfrak b}))
+N\eta\,Y_p((\tfrac{H/F}{\mathfrak b \eta^{-1}})). 
\end{align*}
\end{thm*}

In particular, we see that two refinements are consistent 
(roughly speaking, \cite[Conjecture 3.21]{Da} is a further refinement of \cite[Conjecture A$'$]{KY1} by $\ker \log_p$).
The aim of this paper is to prove this Theorem.

Let us explain the outline of this paper.
In \S 2, we introduce Shintani's technique of cone decompositions.
We obtain a suitable fundamental domain of $F\otimes \mathbb R_+/E_{\mathfrak f,+}$,
where $F\otimes \mathbb R_+$ denotes the totally positive part of $F\otimes \mathbb R$, 
$E_{\mathfrak f,+}$ is a subgroup of the group of all totally positive units.
We need such fundamental domains in order to construct both of the invariants $Y_p$, $u_T$. 
In \S 3, we recall the definition and some properties of $Y_p$, which is essentially defined in \cite{KY1} and slightly modified in \cite{Ka3}.
The classical or $p$-adic $\log$ multiple gamma function is defined as the derivative values at $a=0$ of 
the classical or $p$-adic Barnes' multiple zeta function, respectively.
Then the invariant $Y_p(\tau,\iota)$ is defined in Definition \ref{Yp}, as a finite sum of the ``difference'' of $p$-adic $\log$ multiple gamma functions
and classical $\log$ multiple gamma functions.
Conjecture \ref{KYc} predicts exact values of $Y_p(\tau,\iota)$.
In \S 4, we also recall some results in \cite{Da}. 
Dasgupta introduced $p$-adic measures $\nu_T(\mathfrak b,\mathcal D_\mathfrak f)$ associated with special values of Shintani's multiple zeta functions,
and defined $u_T(\mathfrak b,\mathcal D_\mathfrak f)$ 
as the multiplicative integration $\times\hspace{-0.9em}\int_{\mathbf O}x\ d\nu_\eta(\mathfrak b,\mathcal D_\mathfrak f,x)$ with certain correction terms.
Dasgupta formulated a conjecture (Conjecture \ref{Dc}) on properties of $u_T(\mathfrak b,\mathcal D_\mathfrak f)$.
In \S 5, we state and prove the main result (Theorem \ref{main}) which gives an explicit relation 
between $Y_p(\tau,\mathrm{id})$ and $\log_p(u_\eta(\mathfrak b,\mathcal D_\mathfrak f))$.
Then we will see that Conjectures \ref{KYc}, \ref{Dc} are consistent in the sense of Corollary \ref{crlofmain}.
The key observation is Lemma \ref{key}: Dasgupta's $p$-adic measure $\nu_\eta(\mathfrak b,\mathcal D_\mathfrak f)$ is originally associated with 
Shintani's multiple zeta functions.
By this Lemma, we can relate $\nu_\eta(\mathfrak b,\mathcal D_\mathfrak f)$ to Barnes' multiple zeta functions and $p$-adic analogues 
as in Lemma \ref{intzeta}.

\section{Shintani domains}

Let $F$ be a totally real field of degree $n$, $\mathcal O_F$ the ring of integers of $F$, $\mathfrak f$ an integral ideal of $F$.
We denote by $F_+$ the set of all totally positive elements in $F$ and put $\mathcal O_{F,+}:=\mathcal O_F \cap F_+$,
$E_+:=\mathcal O_F^\times \cap F_+$.
We consider subgroups of $E_+$ of the following form:
\begin{align*}
\efp:=\{\epsilon \in E_+ \mid \epsilon \equiv 1 \bmod \mathfrak f\}.
\end{align*}
We identify
\begin{align*}
F \otimes \mathbb R = \mathbb R^n, \quad \sum_{i=1}^k a_i \otimes b_i \mt (\sum_{i=1}^k\iota(a_i)b_i)_{\iota \in \HFR},
\end{align*}
where $\HFR$ denotes the set of all real embeddings of $F$. 
In particular, the totally positive part
\begin{align*}
F \otimes \mathbb R_+:=\prn^n
\end{align*}
has a meaning. On the right-hand side, $\prn$ denotes the set of all positive real numbers.
Let $v_1,\dots,v_r \in \mathcal O_F$ be linearly independent. Then we define the cone with basis $\bm v:=(v_1,\dots,v_r)$ as
\begin{align*}
C(\bm v):=\{\bm t \ttt \bm v \in F\otimes \mathbb R \mid \bm t \in \prn^r\}.
\end{align*}
Here we $\bm t \ttt \bm v$ denotes the inner product.

\begin{dfn}
\begin{enumerate}
\item We call a subset $D \subset F \otimes \mathbb R_+$ is a Shintani set if it can be expressed as a finite disjoint union of cones:
\begin{align*}
D=\coprod_{i \in J} C(\bm v_j) \quad (|J|<\infty,\ \bm v_j \in \mathcal O_{F,+}^{r(j)},\ r(j) \in \mathbb N).
\end{align*}
\item We consider the natural action $\efp \curvearrowright F \otimes \mathbb R_+$, $u(a\otimes b):=(ua)\otimes b$.
We call a Shintani set $D$ a Shintani domain $\bmod \efp$ if it is a fundamental domain of $ F \otimes \mathbb R_+/\efp$:
\begin{align*}
F \otimes \mathbb R_+=\coprod_{\epsilon \in \efp} \epsilon D.
\end{align*}
\end{enumerate}
When $\mathfrak f=(1)$, we write $\bmod E_+$ instead of $\bmod E_{(1),+}$.
\end{dfn}

Shintani \cite[Proposition 4]{Sh} showed that there always exists a Shintani domain.

\section{$p$-adic $\log$ multiple gamma functions}

We recall the definition and some properties of the symbol $Y_p$ defined in \cite{KY1}, \cite{Ka3}.
We denote by $\prn$ the set of all positive real numbers.

\begin{dfn}
Let $z \in \prn$, $\bm v \in \prn^r$. 
Barnes' multiple zeta function is defined as 
\begin{align*}
\zeta(s,\bm v,z):=\sum_{\bm m \in \nni^r} (z+\bm m\ttt\bm v)^{-s}.
\end{align*}
This series converges for $\Re(s) >r$, has a meromorphic continuation to the whole $s$-plane, and is analytic at $s=0$.
Then Barnes' multiple gamma function is defined as
\begin{align*}
\Gamma(z,\bm v):=\exp\left(\frac{\partial}{\partial s}\zeta(s,\bm v,z)|_{s=0}\right).
\end{align*}
\end{dfn}

Note that this definition is modified from that given by Barnes.
For the proof and details, see \cite[Chap I, \S 1]{Yo}.
Throughout this paper, we regard each number field as a subfield of $\oq$, 
and fix two embeddings $\oq \hookrightarrow \mathbb C$, $\oq \hookrightarrow \mathbb C_p$.
Here $\mathbb C_p$ denotes the $p$-adic completion of the algebraic closure of $\mathbb Q_p$.
We denote by $\mu_{(p)}$ the group of all roots of unity of prime-to-$p$ order.
Let $\mathrm{ord}_p\colon \mathbb C_p^\times \ra \mathbb Q$, $\theta_p \colon \mathbb C_p^\times \ra \mu_{(p)}$ be unique group homomorphisms 
satisfying 
\begin{align} \label{ordtheta}
|p^{-\mathrm{ord}_p(z)}\theta_p(z)^{-1}z|_p<1 \quad (z \in \mathbb C_p^\times).
\end{align}

\begin{dfn}
Let $z \in \oq$, $\bm v \in (\oq^\times)^r$.
We assume that 
\begin{align}
&z \in \prn,\ \bm v \in \prn^r \text{ via the embedding }\oq \hookrightarrow \mathbb C, \notag \\
&\mathrm{ord}_p(z) < \mathrm{ord}_p(v_1),\dots,\mathrm{ord}_p(v_r) \text{ via the embedding }\oq \hookrightarrow \mathbb C_p. \label{cond}
\end{align}
Then we denote by $\zeta_p(s,\bm v,z)$ $(s \in \mathbb Z_p-\{1,2,\dots,r\})$ the $p$-adic multiple zeta function characterized by 
\begin{align} \label{interpolation}
\zeta_p(-m,\bm v,z)=p^{-\mathrm{ord}_p(z)m}\theta_p(z)^{-m}\zeta(-m,\bm v,z) \quad (m \in \nni).
\end{align}
We define the $p$-adic $\log$ multiple gamma function as
\begin{align*}
L\Gamma_p(z,\bm v):=\frac{\partial}{\partial s}\zeta_p(s,\bm v,z)|_{s=0}.
\end{align*}
\end{dfn}

The construction of $\zeta_p(s,\bm v,z)$ is due to Cassou-Nogu\`es \cite{CN1}.
The author defined and studied $L\Gamma_p(z,\bm v)$ in \cite{Ka1}.
See \cite[\S 2]{Ka3} for a short survey.

\begin{dfn}
Let $F$ be a totally real field, $\mathfrak f$ an integral ideal, 
$D=\coprod_{j \in J}C(\bm v_j)$ $(\bm v_j \in \mathcal O_{F,+}^{r(j)})$ a Shintani domain $\bmod E_+$.
We denote by $\HFR$ {\rm (}resp.\ $\HFCp${\rm)} the set of all embeddings of $F$ into $\mathbb R$ {\rm(}resp.\ $\mathbb C_p${\rm )}. 
Since we fixed embeddings $\oq \hookrightarrow \mathbb C$, $\oq \hookrightarrow \mathbb C_p$, we may identify 
\begin{align*}
\HFR=\HFCp.
\end{align*}
\begin{enumerate}
\item We denote by $C_\mathfrak f$ the narrow ideal class group modulo $\mathfrak f$,
by $H_\mathfrak f$ the narrow ray class field modulo $\mathfrak f$.
In particular, the Artin map induces 
\begin{align*}
C_\mathfrak f \cong \mathrm{Gal}(H_\mathfrak f/F).
\end{align*}
\item Let $\pi \colon C_\mathfrak f \ra C_{(1)}$ be the natural projection. 
For each $c \in C_\mathfrak f$, we take an integral ideal $\mathfrak a_c$ satisfying 
\begin{align*}
\mathfrak a_c\mathfrak f \in \pi(c).
\end{align*}
\item For $c \in C_\mathfrak f$, $\bm v \in \mathcal O_F^r$, we put
\begin{align*}
R(c,\bm v):=R(c,\bm v,\mathfrak a_c):=\{\bm x \in (\mathbb Q\cap (0,1])^r \mid \mathcal O_F \supset (\bm x\ttt\bm v) \mathfrak a_c\mathfrak f\in c\}.
\end{align*}
\item For $c \in C_\mathfrak f$, $\iota \in \HFR$, we define
\begin{align*}
G(c,\iota):=G(c,\iota,D,\mathfrak a_c):=\sum_{j \in J} \sum_{\bm x \in R(c,\bm v_j)} \log \Gamma(\iota(\bm x \ttt \bm v_j), \iota(\bm v_j)).
\end{align*}
\item For $\iota \in \HFR$ $(=\HFCp)$, we put
\begin{align*}
\mfpi:=\{z \in \mathcal O_F \mid |\iota(z)|_p<1\}.
\end{align*}
Note that the prime ideal $\iota(\mfpi)$ corresponds to the $p$-adic topology on $\iota(F) \subset \mathbb C_p$.
\item Assume that $\mfpi \mid \mathfrak f$. For $c \in C_\mathfrak f$, $\iota \in \HFR$, we define
\begin{align*}
G_p(c,\iota):=G_p(c,\iota,D,\mathfrak a_c):=\sum_{j \in J} \sum_{\bm x \in R(c,\bm v_j)} L\Gamma_p(\iota(\bm x \ttt \bm v_j), \iota(\bm v_j)).
\end{align*}
Note that $(\iota(\bm x \ttt \bm v_j), \iota(\bm v_j))$ satisfies the assumption {\rm (\ref{cond})} whenever $\mfpi \mid \mathfrak f$, $\bm x \in R(c,\bm v_j)$.
\end{enumerate}
\end{dfn}

The following map $[\ ]_p$ is well-defined by \cite[Lemma 5.1]{KY1}.

\begin{dfn}
We denote by $\oq \log_p \oq^\times$ {\rm (}resp.\ $\oq \log \oq^\times${\rm )} the $\oq$-subspace of $\mathbb C_p$ {\rm (}resp.\ $\mathbb C${\rm )} 
generated by $\log_p b$ {\rm (}resp.\ $\pi$, $\log b${\rm )} with $b \in \oq^\times$.
We define a $\oq$-linear map $[\ ]_p$ by 
\begin{align*}
[\ ]_p \colon \oq \log \oq^\times \ra \oq \log_p \oq^\times, \quad a\log b \mt a\log_p b, \quad a\pi \mt 0\quad (a,b \in \oq,\ b\neq 0).
\end{align*}
\end{dfn}

\begin{lmm}
Let $H$ be an intermediate field of $H_\mathfrak f/F$,
$\mathfrak q$ a prime ideal of $F$, relatively prime to $\mathfrak f$, splitting completely in $H/F$. 
Then we have
\begin{align*}
\sum_{c \in C_{\mathfrak f\mathfrak q},\ \mathrm{Art}(\overline c)|_H=\tau}G(c,\iota) \in \oq\log \oq^\times \quad (\tau \in \mathrm{Gal}(H/F)).
\end{align*}
Here $c$ runs over all ideal classes whose images under the composite map 
$C_{\mathfrak f\mathfrak q} \ra C_\mathfrak f \ra \mathrm{Gal}(H_\mathfrak f/F) \ra \mathrm{Gal}(H/F)$ is equal to $\tau$.
\end{lmm}

\begin{proof}
We put $W(c,\iota):=W(\iota(c))$ in \cite[(4.3)]{KY1}, $V(c,\iota):=V(\iota(c))$ in \cite[(1.6)]{KY1}, and $X(c,\iota):=G(c,\iota)+W(c,\iota)+V(c,\iota)$.
Here we consider the ideal class group $C_{\iota(\mathfrak f)}$ of $\iota(F)$ modulo $\iota(\mathfrak f)$.
Then $\iota(c)$ denotes the image of $c\in C_\mathfrak f$ in $C_{\iota(\mathfrak f)}$ under the natural map.
By the definition \cite[(4.3)]{KY1} and \cite[Appendix I, Theorem]{KY2}, we have $W(c,\iota),V(c,\iota) \in \oq\log \oq^\times$.
Moreover \cite[Lemma 5.5]{KY1} states that 
\begin{align*}
\sum_{c \in C_{\mathfrak f\mathfrak q}} \chi_\mathfrak q(c) X(c,\iota) \in \oq\log \oq^\times \quad (\chi \in \hat C_\mathfrak f,\ \chi([\mathfrak q])=1).
\end{align*}
Here $\chi_\mathfrak q$, $[\mathfrak q]$ denote the composite map $C_{\mathfrak f\mathfrak q} \ra C_\mathfrak f \st{\chi}\ra \mathbb C^\times$,
the ideal class $\in C_\mathfrak f$ of $\mathfrak q$, respectively.
Therefore, when $H$ is the fixed subfield under $(\frac{H_\mathfrak f/F}{\mathfrak q})$, it follows from the orthogonality of characters.
The general case follows from this case immediately.
\end{proof}

\begin{dfn} \label{Yp}
Let $H$ be an intermediate field of $H_\mathfrak f/F$.
Assume that $\mfpi \nmid \mathfrak f$ and that $\mfpi$ splits completely in $H/F$. 
Then we define 
\begin{align*}
Y_p(\tau,\iota):=\sum_{c \in C_{\fpi},\ \mathrm{Art}(\overline c)|_H=\tau}G_p(c,\iota)-[\sum_{c \in C_{\fpi},\ \mathrm{Art}(\overline c)|_H=\tau}G(c,\iota)]_p 
\quad (\tau \in \mathrm{Gal}(H/F)).
\end{align*}
When $\iota=\mathrm{id}$, we drop the symbol $\iota$: $Y_p(\tau):=Y_p(\tau,\mathrm{id})$.
\end{dfn}

By \cite[Proposition 5.6]{KY1} (and the orthogonality of characters), 
we see that $Y_p(\tau,\iota)$ depends only on $H,\mathfrak f,\tau,\iota$, not on $D$, $\mathfrak a_c$'s.
We formulated a conjecture \cite[Conjecture A$'$]{KY1}, which is equivalent to the following Conjecture \ref{KYc} by \cite[Proposition 6-(ii)]{Ka3}.

\begin{cnj} \label{KYc}
Let $H_\mathfrak f/H/F$ be as above: we assume that
\begin{center}
$\mfpi$ does not divide $\mathfrak f$, splits completely in $H/F$. 
\end{center}
We take a lift $\tilde\iota\colon H\ra \mathbb C_p$ of $\iota\colon F\ra \mathbb C_p$
and put $\mathfrak p_{H,\tilde\iota}:=\{z\in \mathcal O_H \mid |\tilde\iota(z)|_p<1\}$.
Let $\alpha_{H,\tilde\iota}$ be a generator of the principal ideal $\mathfrak p_{H,\tilde\iota}^{h_H}$, where $h_H$ denotes the class number.
Then we have
\begin{align*}
Y_p(\tau,\iota)=\frac{-1}{h_H}\sum_{c\in C_\mathfrak f}\zeta(0,c^{-1})\log_p \tilde\iota\left(\alpha_{H,\tilde\iota}^{\tau \mathrm{Art}(c)}\right).
\end{align*}
\end{cnj}

\begin{rmk}
Roughly speaking, the above conjecture states a relation between 
the ratios $[$ $p$-adic multiple gamma functions $:$ multiple gamma functions $]$ and Stark units associated with the finite place $\mathfrak p_\iota$.
We also studied a relation between the same ratios and Stark units associated with real places in {\rm \cite{Ka3}}.
We found a more significant relation 
between the ratios $[$ $p$-adic gamma function $:$ gamma function $]$ and cyclotomic units in {\rm \cite{Ka2}}.
\end{rmk}

We rewrite the definition of $Y_p$ for later use.

\begin{dfn} \label{zetaR}
Let $R$ be a subset of $F_+$. 
We assume that $R$ can be expressed in the following form:
\begin{align*}
R=\coprod_{i=1}^k\{(\bm x_i+\bm m) \ttt \bm v_i \mid \bm m \in \nni^{r_i}\} \quad (\bm x_i \in \mathbb Q_+^{r_i},\ \bm v_i \in F_+^{r_i}).
\end{align*}
\begin{enumerate}
\item We define
\begin{align*}
\zeta_\iota(s,R)&:=\sum_{z \in R} \iota(z)^{-s}:=\sum_{i=1}^k \zeta(s,\iota(\bm v_i),\iota(\bm x_i \ttt \bm v_i)), \\
L\Gamma_\iota(R)&:=\frac{\partial}{\partial s}\zeta_\iota(s,R)|_{s=0}=\sum_{i=1}^k \log \Gamma(\iota(\bm x_i \ttt \bm v_i),\iota(\bm v_i)). 
\end{align*}
\item Additionally we assume that each $(\iota(\bm x_i \ttt \bm v_i),\iota(\bm v_i))$ satisfies {\rm (\ref{cond})}.
Then there exists the $p$-adic interpolation function
\begin{align*}
\zeta_{\iota,p}(s,R):= \sum_{i=1}^k \zeta_p(s,\iota(\bm v_i),\iota(\bm x_i \ttt \bm v_i))
\end{align*}
of $\zeta_{\iota}(s,R)$. We define
\begin{align*}
L\Gamma_{\iota,p}(R):=\frac{\partial}{\partial s}\zeta_{\iota,p}(s,R)|_{s=0}=\sum_{i=1}^k L\Gamma_p(\iota(\bm x_i \ttt \bm v_i),\iota(\bm v_i)).
\end{align*}
\end{enumerate}
When $\iota=\mathrm{id}$, we drop the symbol $\iota$.
\end{dfn}

It follows that, for any Shintani domain $D$ $\bmod E_+$ and for any integral ideals $\mathfrak a_c$ satisfying $\mathfrak a_c\mathfrak f \in \pi(c)$, we have 
\begin{align} \label{rewrite}
Y_p(\tau,\iota)=\sum_{c \in C_{\fpi},\ \mathrm{Art}(\overline c)|_H=\tau}L\Gamma_{\iota,p}(R_c)
-[\sum_{c \in C_{\fpi},\ \mathrm{Art}(\overline c)|_H=\tau}L\Gamma_\iota(R_c)]_p,
\end{align}
where we put $R_c:=\{z \in D \mid \mathcal O_F \supset z \mathfrak a_c \fpi \in c\}$.
We will use the following properties of the classical or $p$-adic multiple gamma functions in the proof of Theorem \ref{main}. 

\begin{prp} \label{eZ}
\begin{enumerate}
\item Let $R$ be as in {\rm Definition \ref{zetaR}-(i)}, $\alpha \in F_+$. Then we have
\begin{align*}
L\Gamma_\iota(R)-L\Gamma_\iota(\alpha R)=\zeta_\iota(0,R)\log \iota(\alpha).
\end{align*}
\item Let $R$ be as in {\rm Definition \ref{zetaR}-(ii)}, $\alpha \in F_+$. Then we have
\begin{align*}
L\Gamma_{\iota,p}(R)-L\Gamma_{\iota,p}(\alpha R)=\zeta_\iota(0,R)\log_p \iota(\alpha).
\end{align*}
\end{enumerate}
\end{prp}

\begin{proof}
The assertions follow from $\zeta_\iota(s,\alpha R)=\iota(\alpha)^{-s}\zeta_\iota(s,R)$ immediately.
\end{proof}

We also recall Shintani's multiple zeta functions in {\rm \cite[(1.1)]{Sh}} which we need in subsequent sections.

\begin{dfn} \label{Sz}
\begin{enumerate}
\item Let $A=(a_{ij})$ be an $(l\times r)$-matrix with $a_{ij} \in \prn$, $\bm x \in \prn^r$, $\bm \chi=(\chi_1,\dots,\chi_r) \in (\mathbb C^\times)^r$ with 
$|\chi_i|\leq 1$. Then Shintani's multiple zeta function is defined as
\begin{align*}
\zeta(s,A,\bm x,\bm \chi):=\sum_{(m_1,\dots,m_r) \in \nni^r} \left(\prod_{j=1}^r\chi_j^{m_j}\right)\left(\prod_{i=1}^l \left(\sum_{j=1}^ra_{ij}(m_j+x_j)\right)\right)^{-s}.
\end{align*}
This series converges for $\Re(s)>\frac{r}{l}$, has a meromorphic continuation to the whole $s$-plane, is analytic at $s=0$.
\item Let $\bm x$, $\bm \chi$ be as in {\rm (i)}. For $\bm v=(v_1,\dots,v_r) \in F_+^r$, we consider two kinds of Shintani's multiple zeta functions: 
\begin{enumerate}
\item Shintani's multiple zeta function with $l=1$:
\begin{align*}
\zeta(s,\bm v,\bm x,\bm \chi)=\sum_{(m_1,\dots,m_r) \in \nni^r}\left(\prod_{j=1}^r\chi_j^{m_j}\right)\left(\sum_{j=1}^r v_j(m_j+x_j)\right)^{-s}.
\end{align*}
Here we consider $v_i \in F_+ \st{\mathrm{id}}\hookrightarrow \prn$.
\item Let $A$ be the $(n\times r)$-matrix whose raw vectors are $\iota(\bm v_i)$ $(\iota \in \HFR,\ n:=[F:\mathbb Q])$. Then we put
\begin{align*}
\zeta_N(s,\bm v,\bm x,\bm \chi):=\zeta(s,A,\bm x,\bm \chi)
=\sum_{(m_1,\dots,m_r) \in \nni^r}\left(\prod_{j=1}^r\chi_j^{m_j}\right)N\left(\sum_{j=1}^r v_j(m_j+x_j)\right)^{-s}.
\end{align*}
\end{enumerate}
\item Let $R$ be as in {\rm Definition \ref{zetaR}-(i)}. We define 
\begin{align*}
\zeta_N(s,R):=\sum_{z \in R} N z^{-s}:=\sum_{i=1}^k \zeta_N(s,\bm v_i,\bm x_i,(1,\dots,1)).
\end{align*}
\end{enumerate}
\end{dfn}

\section{$p$-adic measures associated with zeta values}

We consider the following two kinds of integration.

\begin{dfn}
Let $K$ be a finite extension of $\mathbb Q_p$, $O$ the ring of integers of $K$, $P$ the maximal ideal of $O$. 
\begin{enumerate}
\item We say $\nu$ is a $p$-adic measure on $O$ if for each open compact subset $U \subset O$,
it takes the value $\nu(U)\in K$ satisfying 
\begin{enumerate}
\item $\nu(U\coprod  U')=\nu(U)+\nu(U')$ for disjoint open compact subsets $U,U'$.
\item $|\nu(U)|_p$'s are bounded.
\end{enumerate}
We say a $p$-adic measure $\nu$ is a $\mathbb Z$-valued measure if $\nu(U) \in \mathbb Z$.
\item Let $\nu$ be a $p$-adic measure, $f\colon O\ra O$ a continuous map. We define 
\begin{align*}
\int_Of(x)d\nu(x):=\lim_{\leftarrow}\sum_{\overline a \in O/P^m}\nu(f^{-1}(a+P^m))f(a) \in \lim_{\leftarrow} O/P^m=O.
\end{align*}
\item Let $\nu$ be a $\mathbb Z$-valued measure, $f\colon (O-P^e) \ra (O-P^e)$ a continuous map $(e \in \mathbb N)$. We define 
\begin{align*}
\times\hspace{-1.1em}\int_{O-P^e}f(x)d\nu(x):=\lim_{\leftarrow}\prod_{\overline a\in (O-P^e)/(1+P^m)}f(a)^{\nu(f^{-1}(a+P^m))} \in O.
\end{align*}
\end{enumerate}
\end{dfn}

We recall the setting in \cite{Da}.
Let $F$ be a totally real field of degree $n$, $\mathfrak f$ an integral ideal of $F$, 
$\mathfrak p:=\mathfrak p_{\mathrm{id}}$ the prime ideal corresponding to the $p$-adic topology on $F$ induced by 
$\mathrm{id}\colon F \hookrightarrow \mathbb C_p$.
We assume that $\mathfrak p \nmid \mathfrak f$.

\begin{dfn}[{\cite[Definitions 3.8, 3.9]{Da}}]
Let $\eta$ be a prime ideal of $F$. 
\begin{enumerate}
\item We say $\eta$ is good for a cone $C(v_1,\dots,v_r)$ if $v_i \in \mathcal O_F-\eta$ and if $N\eta$ is a rational prime (i.e., the residue degree $=1$).
\item We say $\eta$ is good for a Shintani set $D$ if it can be expressed as a finite disjoint union of cones for which $\eta$ is good. 
\end{enumerate}
\end{dfn}

\begin{dfn}[{\cite[Definitions 3.13, 3.16, Conjecture 3.21]{Da}}] \label{assump}
We take an element $\pi \in \mathcal O_{F,+}$, a prime ideal $\eta$, a Shintani domain $\mathcal D_\mathfrak f$ $\bmod E_{\mathfrak f,+}$ 
satisfying the following conditions.
\begin{enumerate}
\item Let $e$ be the order of $\mathfrak p$ in $C_\mathfrak f$.
We fix a generator $\pi \in \mathfrak p^e$ satisfying $\pi\in \mathcal O_{F,+}$, $\pi \equiv 1 \bmod \mathfrak f$. 
\item $N\eta \geq n+2$ and $(N\eta,\mathfrak f \mathfrak p)=1$. 
\item The residue degree of $\eta$ $=1$ and the ramification degree of $\eta$ $\leq N\eta-2$. 
\item $\eta$ is ``simultaneously'' good for $\mathcal D_\mathfrak f,\pi^{-1}\mathcal D_\mathfrak f$ in the following sense: 
There exist vectors $\bm v_j \in (\mathcal O_{F,+}- \eta)^{r(j)}$,
units $\epsilon_j\in E_{\mathfrak f,+}$ $(j\in J',\ |J'|<\infty)$ satisfying
\begin{align*}
\mathcal D_\mathfrak f=\coprod_{j\in J'} C(\bm v_j),
\quad \pi^{-1}\mathcal D_\mathfrak f=\coprod_{j\in J'} \epsilon_j C(\bm v_j).
\end{align*}
\end{enumerate}
\end{dfn}

\begin{rmk}
Dasgupta {\rm \cite{Da}} took a suitable set $T$ of prime ideals instead of one prime ideal $\eta$. In this article, we assume that $|T|=1$ for simplicity.
\end{rmk}

We denote by $F_\mathfrak p,\mathcal O_{F_\mathfrak p}$ the completion of $F$ at $\mathfrak p$, 
the ring of integers of $F_\mathfrak p$ respectively.

\begin{dfn}[{\cite[Definitions 3.13, 3.17]{Da}}] \label{ueta}
Let $\pi,\eta,\mathcal D_\mathfrak f$ be as in {\rm Definition \ref{assump}}, 
$\mathfrak b$ a fractional ideal of $F$ relatively prime to $\mathfrak f \mathfrak pN\eta$.
We put 
\begin{align*}
F^\times_\mathfrak f:=\{z\in F^\times\mid z\equiv 1\ {\bmod}^*\ \mathfrak f\}.
\end{align*}
\begin{enumerate}
\item For an open compact subset $\mathcal U \subset \mathcal O_{F_\mathfrak p}$, a Shintani set $D$, we put 
\begin{align*}
\nu(\mathfrak b,D,\mathcal U)
&:=\zeta_N(0,F^\times_\mathfrak f\cap\mathfrak b^{-1}\cap D\cap\mathcal U), \\
\nu_\eta(\mathfrak b,D,\mathcal U)&:=
\nu(\mathfrak b,D,\mathcal U)-N\eta \nu(\mathfrak b\eta^{-1},\mathcal D,\mathcal U).
\end{align*}
Here $\zeta_N(s,R)$ is defined in {\rm Definition \ref{Sz}}.
By {\rm \cite[Proposition 3.12]{Da}} we see that 
\begin{itemize}
\item When $\eta$ is good for $D$, we have $\nu_\eta(\mathfrak b,D,\mathcal U) \in \mathbb Z[N\eta^{-1}]$.
\item  When $\eta$ is good for $D$ and $N\eta \geq n +2$, we have $\nu_\eta(\mathfrak b,D,\mathcal U) \in \mathbb Z$.
\end{itemize}
\item Assume that $\eta$ is good for $D$ and that $\eta \nmid p$. 
We define a $p$-adic measure $\nu_\eta(\mathfrak b,D)$ on $\mathcal O_{F_\mathfrak p}$ by 
\begin{align*}
\nu_\eta(\mathfrak b,D)(\mathcal U):=\nu_\eta(\mathfrak b,D,\mathcal U).
\end{align*}
Under the assumption of {\rm Definition \ref{assump}}, $\nu_\eta(\mathfrak b,\mathcal D_\mathfrak f)$ is a $\mathbb Z$-valued measure.
\item For $\tau\in \mathrm{Gal}(H_\mathfrak f/F)$, we put 
\begin{align*}
\zeta_\mathfrak f(s,\tau)&
:=\sum_{\mathfrak a\subset\mathcal O_F,\ (\frac{H_\mathfrak f/F}{\mathfrak a})=\tau,\ (\mathfrak{a,f})=1} N\mathfrak a^{-s}, \\
\zeta_{\mathfrak f,\eta}(s,\tau)&:=\zeta_\mathfrak f(s,\tau)
-N\eta^{1-s}\zeta_\mathfrak f(s,\tau (\tfrac{H_\mathfrak f/F}{\eta^{-1}} )).
\end{align*}
Here $H_\mathfrak f$ denotes the narrow ray class field modulo $\mathfrak f$.
\item We define
\begin{align*}
\epsilon_\eta(\mathfrak b,\mathcal D_\mathfrak f,\pi)&:=\prod_{\epsilon\in E_{\mathfrak f,+}}
\epsilon^{\nu_\eta(\mathfrak b,\epsilon\mathcal D_\mathfrak f\cap\pi^{-1}\mathcal D_\mathfrak f,\mathcal O_{F_\mathfrak p})} \in E_{\mathfrak f,+}, \\
u_\eta(\mathfrak b,\mathcal D_\mathfrak f)&:=
\epsilon_\eta(\mathfrak b,\mathcal D_\mathfrak f,\pi)
\pi^{\zeta_{\mathfrak f,\eta}(0,(\frac{H_\mathfrak f/F}{\mathfrak b}))}
\times\hspace{-1.1em}\int_{\mathbf O}x\ d\nu_\eta(\mathfrak b,\mathcal D_\mathfrak f,x)\in F_\mathfrak p^\times,
\end{align*}
where $\mathbf O:=\mathcal O_{F_\mathfrak p}-\pi\mathcal O_{F_\mathfrak p}$.
The product in the first line is actually a finite product 
since $\nu_\eta(\mathfrak b,\epsilon\mathcal D_\mathfrak f\cap\pi^{-1}\mathcal D_\mathfrak f,\mathcal O_{F_\mathfrak p})=0$
for all but finite $\epsilon\in E_{\mathfrak f,+}$.
\end{enumerate}
\end{dfn}

\begin{cnj}[{\cite[Conjecture 3.21]{Da}}] \label{Dc}
Let $\pi,\eta,\mathcal D_\mathfrak f$ be as in {\rm Definition \ref{assump}},
$H$ the fixed subfield of $H_\mathfrak f$ under $(\frac{H_\mathfrak f/F}{\mathfrak p})$.
\begin{enumerate}
\item Let $\tau \in \mathrm{Gal}(H/F)$. 
For a fractional ideal $\mathfrak b$ relatively prime to $\mathfrak f \mathfrak pN\eta$ satisfying $(\frac{H/F}{\mathfrak b})=\tau$, we put
\begin{align*}
u_\eta(\tau):=u_\eta(\mathfrak b,\mathcal D_\mathfrak f).
\end{align*}
Then $u_\eta(\tau)$ depends only on $\mathfrak f,\tau,\eta$, not on the choices of $\mathcal D_\mathfrak f,\mathfrak b$.
\item For any $\tau \in \mathrm{Gal}(H/F)$, $u_\eta(\tau)$ is a $\mathfrak p$-unit of $H$ satisfying $u_\eta(\tau)\equiv 1 \bmod \eta$.
\item For any $\tau,\tau' \in \mathrm{Gal}(H/F)$, we have $u_\eta(\tau\tau')=u_\eta(\tau)^{\tau'}$.
\end{enumerate}
\end{cnj}

\section{The main results}

We keep the notation in the previous sections: 
Let $F$ be a totally real field of degree $n$, $H_\mathfrak f$ the narrow ray class field modulo $\mathfrak f$.
We assume that the prime ideal $\mathfrak p$ corresponding to the $p$-adic topology on $F$ does not divide $\mathfrak f$.
Let $H$ be the fixed subfield of $H_\mathfrak f$ under $(\frac{H_\mathfrak f/F}{\mathfrak p})$.
For $\tau\in \mathrm{Gal}(H/F)$, let $Y_p(\tau):=Y_p(\tau,\mathrm{id})$ be as in Definition \ref{Yp}.
For a fractional ideal $\mathfrak b$ relatively prime to $\mathfrak f \mathfrak pN\eta$,
let $u_\eta(\mathfrak b,\mathcal D_\mathfrak f)$ be as in Definition \ref{ueta}-(iv).

\begin{thm} \label{main}
We have
\begin{align*}
\log_p(u_\eta(\mathfrak b,\mathcal D_\mathfrak f))
=-Y_p((\tfrac{H/F}{\mathfrak b}))
+N\eta\,Y_p((\tfrac{H/F}{\mathfrak b\eta^{-1}})) . 
\end{align*}
\end{thm}

\begin{crl} \label{crlofmain}
{\rm Conjecture \ref{KYc}} implies
\begin{align*}
u_\eta(\mathfrak b,\mathcal D_\mathfrak f)^{h_H}
\equiv \prod_{\sigma\in\mathrm{Gal}(H/F)}\alpha_H^{\zeta_{\mathfrak f,\eta}(0,\sigma^{-1}) (\frac{H/F}{\mathfrak b}) \sigma}\bmod \ker\log_p,
\end{align*}
where $\mathfrak p_H$, $h_H$, $\alpha_H$ are the prime ideal of $H$ corresponding to the $p$-adic topology on $H$,
the class number of $H$, a generator of $\mathfrak p_H^{h_H}$.
\end{crl}

We prepare some Lemmas in order prove this Theorem.

\begin{lmm}[{\cite[Th\'eor\`eme 13]{CN2}}] \label{CN}
Let $\bm v \in F_+^r$, $z\in F_+$, $\bm \xi=(\xi_1,\dots,\xi_r)$ with $\xi_i$ roots of unity, $\neq 1$. 
For $k\in \nni$, we have
\begin{align*}
\zeta_N(-k,\bm v,z,\bm \xi)
&=\sum_{\bm m=(m_1,\dots,m_r)\in \mathbb N^r}\frac{\sum_{\bm l=(l_1,\dots,l_r),\ 1\leq l_i\leq m_i}
\begin{Bmatrix}\bm m\\\bm l\end{Bmatrix}N(z-\bm l\ttt \bm v)^k}{(\bm 1-\bm \xi)^{\bm m}}, \\
\zeta(-k,\bm v,z,\bm \xi)
&=\sum_{\bm m=(m_1,\dots,m_r)\in \mathbb N^r}\frac{\sum_{\bm l=(l_1,\dots,l_r),\ 1\leq l_i\leq m_i}
\begin{Bmatrix}\bm m\\\bm l\end{Bmatrix}(z-\bm l\ttt \bm v)^k}{(\bm 1-\bm \xi)^{\bm m}}.
\end{align*}
Here we put $(\bm 1-\bm \xi)^{\bm m}:=\prod_{i=1}^{r}(1-\xi_i)^{m_i}$, 
$\begin{Bmatrix}\bm m\\\bm l\end{Bmatrix}:=\prod_{i=1}^r\left((-1)^{l_i-1}\binom{m_i-1}{l_i-1}\right)$
with the binomial coefficient $\binom{m_i-1}{l_i-1}$.
The sum over $\bm m$ is actually a finite sum since we have
$\sum_{\bm l}\begin{Bmatrix}\bm m\\\bm l\end{Bmatrix}N(z-\bm l\ttt\bm v)^k
=\sum_{\bm l}\begin{Bmatrix}\bm m\\\bm l\end{Bmatrix}(z-\bm l\ttt\bm v)^k=0$
if $m_i$ is large enough.
\end{lmm}

\begin{lmm} \label{key}
Let $\nu_\eta(\mathfrak b,D,\mathcal U)$ be as in {\rm Definition \ref{ueta}-(i)}.
Assume that $\eta$ is good for $D$. Then we have
\begin{align*}
\nu_\eta(\mathfrak b,D,\mathcal U)
=\zeta(0,F^\times_\mathfrak f\cap\mathfrak b^{-1}\cap D\cap\mathcal U)
-N\eta\,\zeta(0,F^\times_\mathfrak f\cap\mathfrak b^{-1}\eta\cap D\cap\mathcal U).
\end{align*}
\end{lmm}

\begin{proof}
It is enough to show the statement when 
\begin{itemize}
\item $D$ is a cone $C(\bm v)$ with $\bm v=(v_1,\dots,v_r)$, $v_i \in \mathcal O_F-\eta$.
\item $\mathcal U$ is of the form $a+\mathfrak p^m \mathcal O_{F_\mathfrak p}$ ($m \in \mathbb N$, $a \in \mathcal O_{F_\mathfrak p}$).
\end{itemize}
Put $R:=F^\times_\mathfrak f \cap\mathfrak b^{-1}\cap C(\bm v) \cap(a+\mathfrak p^m \mathcal O_{F_\mathfrak p})$. 
By definition we have
\begin{align}
\nu_\eta(\mathfrak b,C(\bm v),a+\mathfrak p^m \mathcal O_{F_\mathfrak p})
&=\zeta_N(0,R)-N\eta\,\zeta_N(0,\{z\in R \mid \mathrm{ord}_\eta z>0\}) \notag \\
&=[\sum_{z\in R}Nz^{-s}-N\eta \sum_{z\in R,\ \mathrm{ord}_\eta z>0}Nz^{-s}]_{s=0}. \label{eq1}
\end{align}
Let $L$ be a positive integer satisfying $L\in \mathfrak f\mathfrak p^m\mathfrak b^{-1}$, $(\eta,L)=1$.
Then we have 
\begin{align*}
\sum_{z\in R}Nz^{-s}&=\sum_{\bm x\in R_a}\sum_{\bm m \in \nni^r}N( (\bm x+\bm m)\ttt (L\bm v))^{-s}, \\
R_a&:=\{\bm x \in (\mathbb Q\cap (0,1])^r \mid 
\bm x\ttt (L\bm v) \in F^\times_\mathfrak f \cap\mathfrak b^{-1} \cap(a+\mathfrak p^m \mathcal O_{F_\mathfrak p})\}.
\end{align*}
Since $N\eta$ is a rational prime, the following homomorphism is a surjection.
\begin{align*}
\mathbb Z\rightarrow \mathbb Z/N\eta \cong \mathcal O_F/\eta 
\cong {\mathcal O_F}_{(\eta)}/\eta{\mathcal O_F}_{(\eta)}.
\end{align*}
Here we denote the localization of $\mathcal O_F$ at $\eta$ by ${\mathcal O_F}_{(\eta)}$.
Hence for each $\bm x\in R_a$, there exists an integer $n_{\bm x}$ satisfying 
$\bm x\ttt\bm v \equiv n_{\bm x} \bmod \eta{\mathcal O_F}_{(\eta)}$.
Similarly we take $n_i$ satisfying $Lv_i \equiv n_i \bmod \eta{\mathcal O_F}_{(\eta)}$ and put $\bm n_{L\bm v}:=(n_1,\dots,n_r)$.
Then the following are equivalent:
\begin{align*}
\mathrm{ord}_\eta ((\bm x+\bm m)\ttt (L\bm v))>0 \Leftrightarrow n_{\bm x}+\bm m\ttt \bm n_{L\bm v} \equiv 0 \bmod N\eta.
\end{align*}
Let $\zeta$ be a primitive $N\eta$th root of unity. We put $\xi_{\bm x}:=\zeta^{n_{\bm x}}$, $\xi_i:=\zeta^{n_i}$, $\bm \xi_{L\bm v}:=(\xi_1,\dots,\xi_r)$. 
Note that $\xi_i\neq 1$ for any $i$.
Then we have
\begin{align*}
\sum_{\lambda=1}^{N\eta-1}(\xi_{\bm x}\bm \xi_{L\bm v}^{\bm m})^\lambda=
\begin{cases}
-1 & (\mathrm{ord}_\eta ((\bm x+\bm m)\ttt (L\bm v))=0), \\
N\eta-1 & (\mathrm{ord}_\eta ((\bm x+\bm m)\ttt (L\bm v))>0).
\end{cases}
\end{align*}
Here we put $\bm \xi_{L\bm v}^{\bm m}:=\prod_{i=1}^r \xi_i^{m_i}$.
It follows that
\begin{align}
\sum_{\bm x \in R_a}\sum_{\lambda=1}^{N\eta-1}\xi_{\bm x}^\lambda \zeta_N(s,L\bm v,\bm x,\bm \xi_{L\bm v}^\lambda )&=
\sum_{\bm x \in R_a}\sum_{\bm m \in \nni^r}\sum_{\lambda=1}^{N\eta-1}
(\xi_{\bm x}\bm \xi_{L\bm v}^{\bm m})^\lambda N((\bm x+\bm m)\ttt (L\bm v))^{-s} \notag \\
&=-\sum_{z\in R}Nz^{-s}+N\eta \sum_{z\in Z,\ \mathrm{ord}_\eta z>0}Nz ^{-s}. \label{eq2}
\end{align}
Similarly we obtain
\begin{align}
\sum_{\bm x \in R_a}\sum_{\lambda=1}^{N\eta-1}\xi_{\bm x}^\lambda \zeta(s,L\bm v,\bm x,\bm \xi_{L\bm v}^\lambda )
&=-\sum_{z\in R}z^{-s}+N\eta \sum_{z\in Z,\ \mathrm{ord}_\eta z>0}z ^{-s} \notag \\
&=\zeta(s,R)-N\eta\,\zeta(s,\{z\in R \mid \mathrm{ord}_\eta z>0\}). \label{eq3}
\end{align}
By Lemma \ref{CN}, we have for $\bm x \in R_a$
\begin{align}
\zeta_N(0,L\bm v,\bm x,\bm \xi_{L\bm v}^\lambda )=\zeta(0,L\bm v,\bm x,\bm \xi_{L\bm v}^\lambda ). \label{eq4}
\end{align}
Then the assertion follows from (\ref{eq1}), (\ref{eq2}), (\ref{eq3}), (\ref{eq4}).
\end{proof}

Dasgupta's $p$-adic integration $\int d\nu_\eta(\mathfrak b,\mathcal D_\mathfrak f,x)$ is originally associated with 
special values of multiple zeta functions ``with the norm'' $\zeta_N(\cdots)$. 
By the above Lemma, we can rewrite it in terms of special values of multiple zeta functions ``without the norm'' $\zeta(\cdots)$. 
This observation is one of the main discoveries in this paper.

\begin{lmm} \label{intzeta}
Let $\nu_\eta(\mathfrak b,D,\mathcal U)$, $\mathbf O=\mathcal O_{F_\mathfrak p}-\pi\mathcal O_{F_\mathfrak p}$
be as in {\rm Definition \ref{ueta}}.
Assume that $\eta$ is good for $D$.
Then we have for $k\in \nni$, $s \in \mathbb Z_p$
\begin{align}
\int_{\mathbf O}x^k d\nu_\eta(\mathfrak b, D,x)
&=\zeta(-k,F^\times_\mathfrak f \cap\mathfrak b^{-1}\cap D \cap \mathbf O)
-N\eta\,\zeta(-k,F^\times_\mathfrak f \cap\mathfrak b^{-1}\eta\cap D \cap\mathbf O), \notag \\
\int_{\mathbf O}\langle x\rangle^{-s} d\nu_\eta(\mathfrak b, D,x) 
&=\zeta_p(s,F^\times_\mathfrak f \cap\mathfrak b^{-1}\cap D \cap \mathbf O)
-N\eta\,\zeta_p(s,F^\times_\mathfrak f \cap\mathfrak b^{-1}\eta\cap D \cap\mathbf O). \label{pmzitopm}
\end{align}
Here we put $\langle x\rangle:=p^{-\mathrm{ord}_p x}\theta_p(x)^{-1}x$ by using $\mathrm{ord}_p,\theta_p$ in {\rm (\ref{ordtheta})}.
\end{lmm}

\begin{proof}
It is enough to show the statement when $D=C(\bm v)$ with $\bm v=(v_1,\dots,v_r)$, $v_i \in \mathcal O_F-\eta$.
By definition we can write
\begin{align*}
\int_{\mathbf O}x^k d\nu_\eta(\mathfrak b, C(\bm v),x)
=\lim_{\leftarrow }\sum_{\overline a\in \mathbf O/(1+\mathfrak p^m\mathcal O_{F_\mathfrak p})}
a^k\nu_\eta(\mathfrak b,C(\bm v),a(1+\mathfrak p^m\mathcal O_{F_\mathfrak p})).
\end{align*}
By Lemmas \ref{CN}, \ref{key}, we have
\begin{align*}
a^k\nu_\eta(\mathfrak b,C(\bm v),a(1+\mathfrak p^m\mathcal O_{F_\mathfrak p}))
=-\sum_{\bm m\in \mathbb N^r}\sum_{\lambda=1}^{N\eta-1}\sum_{\bm x \in R_a}
\frac{\xi_{\bm x}^\lambda \sum_{1\leq l_i \leq m_i}\begin{Bmatrix}\bm m\\\bm l\end{Bmatrix}a^k}{(\bm 1-\bm \xi_{L\bm v}^\lambda)^{\bm m}},
\end{align*}
where $L,R_a,\xi_{\bm x},\bm \xi_{L\bm v}$ are as in the proof of Lemma \ref{key}.
On the other hand, by Lemma \ref{CN} again, we obtain
\begin{align*}
&\zeta(-k,F^\times_\mathfrak f \cap\mathfrak b^{-1}\cap C(\bm v)\cap\mathbf O)
-N\eta\,\zeta(-k,F^\times_\mathfrak f \cap\mathfrak b^{-1}\eta\cap C(\bm v)\cap\mathbf O) \\
&=-\sum_{\overline a \in \mathbf O/(1+\mathfrak p^m\mathcal O_{F_\mathfrak p})}\sum_{\bm x \in R_a}
\sum_{\bm m \in \mathbb N^r}\sum_{\lambda=1}^{N\eta-1}\frac{\xi_{\bm x}^\lambda \sum_{1\leq l_i \leq m_i}
\begin{Bmatrix}\bm m\\\bm l\end{Bmatrix}((\bm x-\bm l)\ttt (L\bm v))^k}{(\bm 1-\bm \xi_{L\bm v}^\lambda )^{\bm m}}.
\end{align*}
By definition, we see that $L\in \mathfrak p^m$, $\bm x\ttt (L\bm v) \equiv a \bmod \mathfrak p^m\mathcal O_{F_\mathfrak p}$ for $\bm x \in R_a$.
It follows that 
\begin{align*}
a ^k \equiv ((\bm x-\bm l)\ttt (L\bm v))^k \bmod \mathfrak p^m \mathcal O_{F_\mathfrak p} \quad (\bm x \in R_a),
\end{align*}
Hence the first assertion is clear.
The second assertion follows from the $p$-adic interpolation property (\ref{interpolation}).
\end{proof}

\begin{proof}[Proof of {\rm Theorem \ref{main}}]
For a fractional ideal $\mathfrak b$, a Shintani set $D$, an open compact subset $\mathcal U \subset \mathcal O_{F_\mathfrak p}$, and for $*=\emptyset,p$,
we put 
\begin{align*}
L\Gamma_*(\mathfrak b,D,\mathcal U)&:=L\Gamma_*(F^\times_\mathfrak f \cap\mathfrak b^{-1}\cap D\cap\mathcal U), \\
L\Gamma_{\eta,*} (\mathfrak b,D,\mathcal U)
&:=L\Gamma_*(\mathfrak b,D,\mathcal U)-N\eta\,L\Gamma_*(\mathfrak b\eta^{-1},D,\mathcal U)
\end{align*}
whenever each function is well-defined.
It suffices to show the following three equalities:
\begin{align}
\log_p(\epsilon_\eta (\mathfrak b,\mathcal D_\mathfrak f ,\pi)
\pi^{\zeta_{\mathfrak f,\eta}(0,(\frac{H_\mathfrak f/F}{\mathfrak b}))})
&=[L\Gamma_\eta (\mathfrak b,\mathcal D_\mathfrak f ,\mathbf O)]_p, \label{Cl1} \\
\log_p(\times\hspace{-1.1em}\int_{\mathbf O}x\ d\nu_\eta(\mathfrak b,\mathcal D_\mathfrak f ,x))
&=-L\Gamma_{\eta,p}(\mathfrak b,\mathcal D_\mathfrak f ,\mathbf O), \label{Cl2} \\
L\Gamma_p(\mathfrak b,\mathcal D_\mathfrak f ,\mathbf O)-[L\Gamma(\mathfrak b,\mathcal D_\mathfrak f ,\mathbf O)]_p
&=Y_p((\tfrac{H/F}{\mathfrak b})). \label{Cl3}
\end{align}
Let $\bm v_j$, $\epsilon_j$ ($j \in J'$) be as in Definition \ref{assump}-(iv).
Since $\mathcal D_\mathfrak f,\pi^{-1}\mathcal D_\mathfrak f$ are fundamental domains of $F\otimes \mathbb R_+/E_{\mathfrak f,+}$, 
we see that 
\begin{align*}
\epsilon\mathcal D_\mathfrak f\cap\pi^{-1}\mathcal D_\mathfrak f =\coprod_{j \in J',\ \epsilon_j =\epsilon} \epsilon_j C(\bm v_j)
\quad (\epsilon \in E_{\mathfrak f,+})
\end{align*}
Namely we have
\begin{align*}
\epsilon_\eta (\mathfrak b,\mathcal D_\mathfrak f ,\pi)=
\prod_{j\in J'}\epsilon_j ^{\nu_\eta(\mathfrak b,\epsilon_jC(\bm v_j),\mathcal O_{F_\mathfrak p})}.
\end{align*}
 By Lemma \ref{key}, we can write
\begin{align*}
&\nu_\eta(\mathfrak b,\epsilon_jC(\bm v_j),\mathcal O_{F_\mathfrak p}) \\
&=\zeta(0,F^\times_\mathfrak f \cap\mathfrak b^{-1}\cap\epsilon_j C(\bm v_j)\cap\mathcal O_{F_\mathfrak p})-
N\eta\,\zeta(0,F^\times_\mathfrak f \cap\mathfrak b^{-1}\eta\cap\epsilon_j C(\bm v_j)\cap
\mathcal O_{F_\mathfrak p}).
\end{align*}
Therefore by Proposition \ref{eZ}-(i) we obtain
\begin{align} \label{lg1}
\log_p(\epsilon_\eta (\mathfrak b,\mathcal D_\mathfrak f ,\pi))
&=[\sum_{j\in J'}L\Gamma_\eta (\mathfrak b,C(\bm v_j),\mathcal O_{F_\mathfrak p})
-\sum_{j\in J'}L\Gamma_\eta (\mathfrak b,\epsilon_jC(\bm v_j) ,\mathcal O_{F_\mathfrak p})]_p \notag \\
&=[L\Gamma_\eta (\mathfrak b,\mathcal D_\mathfrak f ,\mathcal O_{F_\mathfrak p})
-L\Gamma_\eta (\mathfrak b,\pi^{-1}\mathcal D_\mathfrak f ,\mathcal O_{F_\mathfrak p})]_p.
\end{align}
We easily see that 
\begin{align*}
\zeta_{\mathfrak f,\eta}(0,(\tfrac{H_\mathfrak f/F}{\mathfrak b}))
&=\nu_\eta(\mathfrak b,\mathcal D_\mathfrak f ,\mathcal O_{F_\mathfrak p}), \\
\pi(F^\times_\mathfrak f \cap\mathfrak b^{-1}\cap\pi^{-1}\mathcal D_\mathfrak f \cap\mathcal O_{F_\mathfrak p})
&=F^\times_\mathfrak f \cap\pi\mathfrak b^{-1}\cap\mathcal D_\mathfrak f \cap\pi\mathcal O_{F_\mathfrak p}.
\end{align*}
Hence, by Proposition \ref{eZ}-(i) again, we get
\begin{align} \label{lg2}
\log_p(\pi^{\zeta_{\mathfrak f,\eta}(0,(\frac{H_\mathfrak f/F}{\mathfrak b}))})
=[L\Gamma_\eta (\mathfrak b,\pi^{-1}\mathcal D_\mathfrak f ,\mathcal O_{F_\mathfrak p})
-L\Gamma_\eta (\pi^{-1}\mathfrak b,\mathcal D_\mathfrak f ,\pi\mathcal O_{F_\mathfrak p})]_p.
\end{align}
Since 
$(F^\times_\mathfrak f \cap\mathfrak b^{-1}\cap\mathcal D_\mathfrak f \cap\mathbf O)
\coprod (F^\times_\mathfrak f \cap\pi\mathfrak b^{-1}\cap\mathcal D_\mathfrak f \cap\pi\mathcal O_{F_\mathfrak p})
=F^\times_\mathfrak f \cap\mathfrak b^{-1}\cap\mathcal D_\mathfrak f \cap\mathcal O_{F_\mathfrak p}$,
we have
\begin{align} \label{lg3}
L\Gamma_\eta (\mathfrak b,\mathcal D_\mathfrak f ,\mathbf O)
=L\Gamma_\eta (\mathfrak b,\mathcal D_\mathfrak f ,\mathcal O_{F_\mathfrak p})
-L\Gamma_\eta (\pi^{-1}\mathfrak b,\mathcal D_\mathfrak f ,\pi\mathcal O_{F_\mathfrak p}).
\end{align}
Then the assertion (\ref{Cl1}) follows from (\ref{lg1}), (\ref{lg2}), (\ref{lg3}).

Next, differentiating (\ref{pmzitopm}) at $s=0$, we obtain
\begin{align*}
-\int_{\mathbf O}\log_p x\, d\nu_\eta(\mathfrak b,\mathcal D_\mathfrak f ,x)
=L\Gamma_{\eta,p}(\mathfrak b,\mathcal D_\mathfrak f ,\mathbf O).
\end{align*}
By definition, we have $\log_p(\times\hspace{-0.9em}\int_{\mathbf O}x\ d\nu_\eta(\mathfrak b,\mathcal D_\mathfrak f ,x))
=\int_{\mathbf O}\log_p x\, d\nu_\eta(\mathfrak b,\mathcal D_\mathfrak f ,x)$. Hence the assertion (\ref{Cl2}) is clear. 

Finally we prove (\ref{Cl3}). Let $D$ be a Shintani domain $\bmod E_+$.
For each $c \in C_{\mathfrak f \mathfrak p}$, we take an integral ideal $\mathfrak a_c$ satisfying $\mathfrak a_c\mathfrak f \in \pi(c)$, 
and put $R_c:=\{z \in D \mid \mathcal O_F \supset z \mathfrak a_c \mathfrak f\mathfrak p \in c\}$. By (\ref{rewrite})  we can write 
\begin{align*} 
Y_p((\tfrac{H/F}{\mathfrak b}))=\sum_{c \in C_{\mathfrak f\mathfrak p},\ \mathrm{Art}(\overline c)|_H=(\frac{H/F}{\mathfrak b})}L\Gamma_p(R_c)
-[\sum_{c \in C_{\mathfrak f\mathfrak p},\ \mathrm{Art}(\overline c)|_H=(\frac{H/F}{\mathfrak b})}L\Gamma(R_c)]_p.
\end{align*}
Since $H$ is the fixed subfield under $(\frac{H_\mathfrak f/F}{\mathfrak p})$, 
we may replace 
\begin{align*}
\sum_{c \in C_{\mathfrak f\mathfrak p},\ \mathrm{Art}(\overline c)|_H=(\frac{H/F}{\mathfrak b})}\cdots=
\sum_{k=0}^{e-1}\sum_{c \in C_{\mathfrak f\mathfrak p},\ \overline c=[\mathfrak b\mathfrak p^{-k}]}\cdots,
\end{align*}
where $\overline c$ denotes the image under $C_{\mathfrak f\mathfrak p}\ra C_\mathfrak f$, 
$[\mathfrak a]$ denotes the ideal class in $C_\mathfrak f$ of a fractional ideal $\mathfrak a$.
On the other hand, we can write for $*=\emptyset,p$
\begin{align*}
\begin{split}
L\Gamma_*(\mathfrak b,\mathcal D_\mathfrak f ,\mathbf O)=
\sum_{k=0}^{e-1}L\Gamma_*(\mathfrak b,\mathcal D_\mathfrak f ,\mathfrak p^k\mathcal O_{F_\mathfrak p}^\times).
\end{split}
\end{align*}
Therefore it suffices to show that we have for each $k$ 
\begin{align}
&(\sum_{c \in C_{\mathfrak f\mathfrak p},\ \overline c=[\mathfrak b\mathfrak p^{-k}]}L\Gamma_p(R_c))
-L\Gamma_p(\mathfrak b,\mathcal D_\mathfrak f ,\mathfrak p^k\mathcal O_{F_\mathfrak p}^\times) \notag \\
&=[(\sum_{c \in C_{\mathfrak f\mathfrak p},\ \overline c=[\mathfrak b\mathfrak p^{-k}]}L\Gamma(R_c))
-L\Gamma(\mathfrak b,\mathcal D_\mathfrak f ,\mathfrak p^k\mathcal O_{F_\mathfrak p}^\times)]_p. \label{Cl32}
\end{align}
We fix $k$. Whenever $\overline c=[\mathfrak b\mathfrak p^{-k}]$, $\pi(c) \in C_{(1)}$ is constant,
so we may put $\mathfrak a_c$ to be a fixed integral ideal $\mathfrak a_0$.
Then we have
\begin{align*}
\coprod_{c \in C_{\mathfrak f\mathfrak p},\ \overline c=[\mathfrak b\mathfrak p^{-k}]}R_c=
\{z \in (\mathfrak a_0 \mathfrak f\mathfrak p)^{-1} \cap D 
\mid (z \mathfrak a_0 \mathfrak f\mathfrak p, \mathfrak f\mathfrak p)=1,\ [z\mathfrak a_0 \mathfrak f\mathfrak p]= [\mathfrak b\mathfrak p^{-k}] 
\text{ in } C_\mathfrak f\}.
\end{align*}
Let $\alpha_0 \in F_+$ be a generator of the principal ideal $(\mathfrak a_0\mathfrak f\mathfrak p)(\mathfrak b\mathfrak p^{-k})^{-1}$.
Then the following are equivalent:
\begin{center}
$[z\mathfrak a_0 \mathfrak f\mathfrak p]= [\mathfrak b\mathfrak p^{-k}] \LR [(z\alpha_0)]=[(1)] 
\LR \exists \epsilon \in E_+$ s.t.\ $z\epsilon\alpha_0 \equiv 1 \bmod \mathfrak f$.  
\end{center}
Hence, taking a representative set $E_0$ of $E_+/E_{\mathfrak f ,+}$, we can write 
\begin{align*}
&\{z \in (\mathfrak a_0 \mathfrak f\mathfrak p)^{-1} \cap D 
\mid (z \mathfrak a_0 \mathfrak f\mathfrak p, \mathfrak f\mathfrak p)=1,\ [z\mathfrak a_0 \mathfrak f\mathfrak p]= [\mathfrak b\mathfrak p^{-k}] 
\text{ in } C_\mathfrak f\} \\
&=\coprod_{\epsilon \in E_0}(\epsilon\alpha_0)^{-1}
(F_\mathfrak f ^\times\cap\mathfrak b^{-1}\cap\epsilon\alpha_0D\cap \mathfrak p^k\mathcal O_{F_\mathfrak p}^\times).
\end{align*}
Namely we have  for $*=\emptyset,p$
\begin{align}  
\sum_{c \in C_{\mathfrak f\mathfrak p},\ \overline c=[\mathfrak b\mathfrak p^{-k}]}L\Gamma_*(R_c)
=\sum_{\epsilon \in E_0}L\Gamma_*((\epsilon\alpha_0)^{-1}
(F_\mathfrak f ^\times\cap\mathfrak b^{-1}\cap\epsilon\alpha_0D\cap \mathfrak p^k\mathcal O_{F_\mathfrak p}^\times)). \label{YpR0}
\end{align}
On the other hand, $\mathcal D_\mathfrak f ':=\coprod_{\epsilon \in E_0} \epsilon\alpha_0 D$ becomes another Shintani domain $\bmod E_{\mathfrak f,+}$,
and we can write for $*=\emptyset,p$
\begin{align} \label{R0lg}
L\Gamma_*(\mathfrak b,\mathcal D_\mathfrak f ',\mathfrak p^k\mathcal O_{F_\mathfrak p}^\times)=\sum_{\epsilon \in E_0}
L\Gamma_*(F_\mathfrak f ^\times\cap\mathfrak b^{-1}\cap\epsilon\alpha_0D\cap \mathfrak p^k\mathcal O_{F_\mathfrak p}^\times).
\end{align}
Then the assertion (\ref{Cl32}), replacing $\mathcal D_\mathfrak f$ with $\mathcal D_\mathfrak f'$, follows from (\ref{YpR0}), (\ref{R0lg})
and Proposition \ref{eZ}. 
We conclude the proof of (\ref{Cl3}) by showing that 
\begin{align*}
L\Gamma_p(\mathfrak b,\mathcal D_\mathfrak f ,\mathfrak p^k\mathcal O_{F_\mathfrak p}^\times)
-L\Gamma_p(\mathfrak b,\mathcal D_\mathfrak f ',\mathfrak p^k\mathcal O_{F_\mathfrak p}^\times)
=[L\Gamma(\mathfrak b,\mathcal D_\mathfrak f ,\mathfrak p^k\mathcal O_{F_\mathfrak p}^\times)
-L\Gamma(\mathfrak b,\mathcal D_\mathfrak f ',\mathfrak p^k\mathcal O_{F_\mathfrak p}^\times)]_p.
\end{align*}
Note that the independence on the choice of $\mathcal D_\mathfrak f$ is also discussed in \cite[\S 5.2]{Da} under certain conditions.
Similarly to \cite[Chap.\ III, Lemma 3.13]{Yo},
we see that there exist cones $C(\bm v_j)$ and units $u_j \in E_{\mathfrak f,+}$ ($j\in J''$) which satisfy
\begin{align*}
\mathcal D_\mathfrak f =\coprod_{j\in J''} C(\bm v_j),\quad 
\mathcal D_\mathfrak f '=\coprod_{j\in J''} u_jC(\bm v_j).
\end{align*}
Therefore it suffices to show that
\begin{align*}
&L\Gamma_p(\mathfrak b,C(\bm v_j),\mathfrak p^k\mathcal O_{F_\mathfrak p}^\times)
-L\Gamma_p(\mathfrak b,u_jC(\bm v_j),\mathfrak p^k\mathcal O_{F_\mathfrak p}^\times) \\
&=[L\Gamma(\mathfrak b,C(\bm v_j),\mathfrak p^k\mathcal O_{F_\mathfrak p}^\times)
-L\Gamma(\mathfrak b,u_jC(\bm v_j),\mathfrak p^k\mathcal O_{F_\mathfrak p}^\times)]_p.
\end{align*}
It follows from Proposition \ref{eZ}
since $F_\mathfrak f ^\times\cap\mathfrak b^{-1}\cap u_jC(\bm v_j)\cap\mathfrak p^k\mathcal O_{F_\mathfrak p}^\times
=u_j(F_\mathfrak f ^\times\cap\mathfrak b^{-1}\cap C(\bm v_j)\cap\mathfrak p^k\mathcal O_{F_\mathfrak p}^\times)$.
\end{proof}


\begin{thebibliography}{99}

\bibitem[CN1]{CN1} P.~Cassou-Nogu\`es, Analogues $p$-adiques de quelques fonctions arithm\'etiques, 
\textit{Publ.\ Math.\ Bordeaux} (1974-1975), 1--43.

\bibitem[CN2]{CN2} P.~Cassou-Nogu\`es, Valeurs aux entiers n\'egative des fonction z\^eta et fonction z\^eta $p$-adiques, 
\textit{Inv.\ Math.}
\textbf{51} (1979), 29--59.

\bibitem[Da]{Da} S.~Dasgupta, Shintani zeta-functions and Gross-Stark units for totally real fields, 
\textit{Duke Math. J.} 
\textbf{143} (2008), no.\ 2, 225--279.

\bibitem[DDP]{DDP} S.~Dasgupta, H.~Darmon, R.~Pollack, Hilbert modular forms and the Gross-Stark conjecture, 
\textit{Ann.\ of Math.\ (2)}
\textbf{174} (2011), no.\ 1, 439--484 .

\bibitem[Gr]{Gr} B.~H.~Gross, $p$-adic $L$-series at $s=0$, 
\textit{J.\ Fac.\ Sci.\ Univ.\ Tokyo}
\textbf{28} (1981), 979--994.

\bibitem[Ka1]{Ka1} T.~Kashio, On a $p$-adic analogue of Shintani's formula, 
\textit{J.\ Math.\ Kyoto Univ.}
\textbf{45} (2005), 99--128.

\bibitem[Ka2]{Ka2} T.~Kashio, Fermat curves and a refinement of the reciprocity law on cyclotomic units, 
\textit{J.\ Reine Angew.\ Math.},
DOI: \verb|10.1515/crelle-2015-0081|.

\bibitem[Ka3]{Ka3} T.~Kashio, On the ratios of Barnes' multiple gamma functions to the $p$-adic analogues (arXiv:1703.10411),

\bibitem[KY1]{KY1} T.~Kashio, and H.~Yoshida, On $p$-adic absolute CM-Periods, I, 
\textit{Amer.\ J.\ Math.} 
\textbf{130} (2008), no.\ 6, 1629--1685.

\bibitem[KY2]{KY2} T.~Kashio, and H.~Yoshida, On $p$-adic absolute CM-Periods, II, 
\textit{Publ.\ Res.\ Inst.\ Math.\ Sci.} 
\textbf{45} (2009), no.\ 1, 187--225.

\bibitem[Sh]{Sh} T.~Shintani, On evaluation of zeta functions of totally real algebraic number fields at non-positive integers,
\textit{J.\ Fac.\ Sci.\ Univ.\ Tokyo Sect.\ IA Math.} 
\textbf{23} (1976), no.\ 2, 393--417.

\bibitem[Yo]{Yo} H.~Yoshida, 
\textit{Absolute CM-Periods},
Math.\ Surveys Monogr.\ 
\textbf{106}, Amer.\ Math.\ Soc., Providence, RI, 2003.

\end{thebibliography}
\end{document}